\documentclass{amsart}

\usepackage{amssymb}
\usepackage{graphicx}
\usepackage{MnSymbol}

\usepackage{amsthm}

\title{Limiting theories of substructures}
\author{Samuel M. Corson}

\theoremstyle{definition}\newtheorem*{A}{Theorem \ref{manyautomorphisms}}

\theoremstyle{definition}\newtheorem{theorem}{Theorem}[section]

\theoremstyle{definition}\newtheorem{corollary}[theorem]{Corollary}
\theoremstyle{definition}\newtheorem{proposition}[theorem]{Proposition}
\theoremstyle{definition}
\theoremstyle{definition}\newtheorem{question}[theorem]{Question}
\theoremstyle{definition}\newtheorem{example}{Example}
\theoremstyle{definition}
\theoremstyle{definition}
\theoremstyle{definition}\newtheorem{lemma}[theorem]{Lemma}
\theoremstyle{definition}
\theoremstyle{definition}
\theoremstyle{definition}
\theoremstyle{definition}
\newtheorem*{question*}{Question}
\newtheorem*{theorem*}{Theorem}
\newtheorem*{corollary*}{Corollary}
\newtheorem*{lemma*}{Lemma}
\newtheorem*{fact*}{Fact}

\def\pmc#1{\setbox0=\hbox{#1}
    \kern-.1em\copy0\kern-\wd0
    \kern.1em\copy0\kern-\wd0}

\newcommand{\Th}{\operatorname{Th}}
\newcommand{\Sent}{\operatorname{Sent}}

\newcommand{\W}{\operatorname{Wff}}

\newcommand{\So}{\mathcal{S}}
\newcommand{\card}{\operatorname{card}}

\newcommand{\supp}{\operatorname{supp}}

\begin{document}

\address{Instituto de Ciencias Matem\'aticas CSIC-UAM-UC3M-UCM, 28049 Madrid, Spain.}
\email{sammyc973@gmail.com}

\maketitle

\begin{abstract}  We introduce the notion of limiting theories, giving examples and providing a sufficient condition under which the first order theory of a structure is the limit of the first order theories of a collection of substructures.  We also give a new proof that theories like that of infinite sets are not finitely axiomatizable.
\end{abstract}

\begin{section}{Introduction}

The notions of elementary equivalence and elementary substructure have been fundamental in model theory (\cite{M}, \cite{B1} , \cite{T}, \cite{B2},  \cite{S1}, \cite{KM}, \cite{S2}, \cite{CRK}).  When considering a directed system of substructures whose union is the structure itself, it may be the case that the theory of the large structure is different from that of the smaller structures for obvious reasons.  For example, the structure might be infinite and the substructures under consideration might each be finite.  However, the theories of the substructures might approach the theory of the structure in a way which we will now make precise.

Let $\sigma =(\mathcal{F}, \mathcal{R})$ be a signature, with $\mathcal{F}$ the collection of function symbols and $\mathcal{R}$ the collection of relation symbols.  Let $L_{\sigma}$ be the language associated with $\sigma$, $\W(\sigma)$ the set of well-formed formulas and $\Sent(\sigma) \subseteq \W(\sigma)$ the set of sentences.  Recall that a \textbf{directed set} $I$ is a partially ordered set such that for any $i_1, i_2\in I$ there exists $i_3\in I$ with $i_3 \geq i_1, i_2$.  We will assume throughout this note that every index set labeled by a symbol $I$ is a directed set.  Recall that a \textbf{theory} is a set $\Delta \subseteq \Sent(\sigma)$ such that $\Delta \vdash \theta$ implies $\theta \in \Delta$.  For any $\Gamma \subseteq \Sent(\sigma)$ we let $\Th(\Gamma)$ denote the theory which is the set of consequences of $\Gamma$.  We consider a collection $\{\Delta_i\}_{i\in I}$ of theories in $\sigma$.  Let

\begin{center}
$\limsup_I \Delta_i = \{\theta \in \Sent(\sigma): (\forall i\in I)(\exists j \geq i)[\theta \in \Delta_j]\}$

$\liminf_I \Delta_i = \{\theta \in \Sent(\sigma): (\exists i\in I)(\forall j \geq i)[\theta \in \Delta_j]\}$
\end{center}

The set $\liminf_I \Delta_i$ is itself a theory, $\limsup_{i\in I}\Delta_i$ need not be a theory, and it is clear that  $\limsup_I \Delta_i \supseteq \liminf_I \Delta_i$.  We write $\lim_I\Delta_i = \Delta$ and say the limit exists provided $\Delta = \limsup_I \Delta_i = \liminf_I \Delta_i$.

Let $\So = (S, F, R)$ be a structure of signature $\sigma$.  We consider the situation for limiting theories among substructures.  Suppose $\So$ is the direct limit of a collection of substructures, in the sense that $\{\So_i\}_{i\in I}$ satisfies

\begin{enumerate}
\item $\So_i = (S_i, F|S_i, R|S_i)$

\item $S = \bigcup_{i\in I}S_i$

\item $i \leq j  \Leftrightarrow S_i \subseteq S_j$
\end{enumerate}

\noindent Let $L_{\sigma, \So^*}$ denote the augmented language over $\sigma$ which includes all the constants in the set $S$ and let $\W(\sigma, \So^*)$ and $\Sent(\sigma, \So^*)$ denote the sets of well-formed formulas and of sentences over the language $L_{\sigma, \So^*}$.  Define $L_{\sigma, \So_i^*}$, $\W(\sigma, \So_i^*)$ and $\Sent(\sigma, \So_i^*)$ similarly.  We have the obvious relationships

\begin{center}

$L_{\sigma, \So^*} \supseteq L_{\sigma, \So_i^*}$

$\W(\sigma, \So^*) \supseteq \W(\sigma, \So_i^*)$

$\Sent(\sigma, \So^*)\supseteq \Sent(\sigma, \So_i^*)$

$\Sent(\sigma, \So^*) = \bigcup_{i\in I}\Sent(\sigma, \So_i^*)$
\end{center}

We will overload notation and let $\Th(\So)$ denote the first order theory of $\So$ and $\Th(\So^*)$ denote the augmented first order theory of $\So$.  One can ask under what circumstances $\lim_I \Th(\So_i)$ or $\lim_I\Th(\So_i^*)$ exist, or further ask whether $\lim_I \Th(\So_i)= \Th(\So)$ or more strongly $\lim_I \Th(\So_i^*)= \Th(\So^*)$.  In case $\lim_I \Th(\So_i)= \Th(\So)$ the intuition is that the $\So_i$ come closer and closer to being elementarily equivalent to $\So$.  In case $\lim_I \Th(\So_i^*)= \Th(\So^*)$ the intuition is that the $\So_i$ come closer and closer to capturing the augmented theory of $\So$.

In Section \ref{prelim} we give some straightforward preliminary results on abstract limits of theories.  In Section \ref{directlim} we approach the theory of substructures.  Some easy results and examples are exhibited.  We then prove the following, which can be considered a limit analogue to the Tarski-Vaught automorphism criterion for elementary subsumption:

\begin{theorem}\label{manyautomorphisms}    Suppose $\{\So_i\}_{i\in I}$ is a collection of substructures of $\So$ such that $\lim_I \So_i = \So$.
\begin{enumerate}

\item  Suppose that for each $m \in \mathbb{N}$ and $i' \in I$ there exists $i \geq i'$ such that for all $j \geq i$ and $(m_1, m_2) \in \mathbb{N}^2$ such that $m_1 + m_2 \leq m$ and $c_1, \ldots, c_{m_0} \in S_i$ and $d_1, \ldots, d_{m_2} \in S_j$ there exists an automorphism $f: \So_j \rightarrow \So_j$ with $f(c_l) = c_l$ and $f(\{d_1, \ldots, d_{m_2}\}) \subseteq S_i$.  Then the limit $\lim_I \Th(\So_j^*)$ exists.

\item  If, in addition to (1), for each $m \in \mathbb{N}$ and $i' \in I$ there exists $i \geq i'$ such that whenever $(m_1, m_2) \in \mathbb{N}^2$ are such that $m_1 + m_2 \leq m$ and $c_1, \ldots, c_{m_0} \in S_i$ and $d_1, \ldots, d_{m_2} \in S$ there exists an automorphism $f: \So \rightarrow \So$ with $f(c_l) = c_l$ and $f(\{d_1, \ldots, d_{m_2}\}) \subseteq S_i$.  Then $\lim_I \Th(\So_j^*) = \Th(\So^*)$.

\end{enumerate}

\end{theorem}

This theorem combined with a result in Section \ref{prelim} provides yet another proof that the property of being infinite is not finitely axiomatizable (see Corollary \ref{notfinitelyax}).  The idea is the following.  The theory of an infinite set can be shown to be the limit of the theories of increasingly large finite sets.  If the theory of an infinite set were finitely axiomatizable then the theories of increasingly large finite sets would eventually stabilize, and they obviously do not.  That the theory of an infinite rank free abelian group is not finitely axiomatizable modulo the theory of torsion-free abelian groups follows from a similar argument.
\end{section}

\begin{section}{Preliminaries}\label{prelim}

Recall that a subset $\Gamma \subseteq \Sent(\sigma)$ is \textbf{consistent} if $\Th(\Gamma) \neq \Sent(\sigma)$, \textbf{complete} if for all $\theta \in \Sent(\sigma)$ we have at least one of $\theta$ or $\neg\theta$ in $\Th(\Gamma)$ and \textbf{inconsistent} or \textbf{incomplete} provided it is not consistent or not complete, respectively.

\begin{lemma}\label{trivial3}  Let $\{\Delta_i\}_{i\in I}$ be a directed set of theories.  If each of the $\Delta_i$ is consistent and complete then $\liminf_I \Delta_i$ is consistent and $\limsup_I \Delta_i$ is complete.  Furthermore, the following become equivalent:

\begin{enumerate}

\item  $\limsup_I\Delta_i$ is consistent

\item  $\liminf_I \Delta_i$ is complete

\item  $\lim_I \Delta_i$ exists
\end{enumerate}
\end{lemma}

\begin{proof}  Let $\theta \in \Sent(\sigma)$ be given.  We know $\theta \wedge \neg \theta$ is not in any $\Delta_i$ (since they are each consistent), and therefore $\theta \wedge \neg \theta \notin \liminf_I \Delta_i = \Th(\liminf_I \Delta_i)$ and so $\liminf_I \Delta_i$ is consistent.  By completeness of each $\Delta_i$ we know that at least one of $\theta$ or $\neg \theta$ is in $\limsup_I \Delta_i$ and so $\theta \in \limsup_I \Delta_i \subseteq \Th(\limsup_I \Delta_i)$ or $\neg \theta \in \limsup_I \Delta_i \subseteq \Th(\limsup_I \Delta_i)$ and $\limsup_I \Delta_i$ is complete.

We prove the claim in the third sentence.

(1) $\Rightarrow$ (2) If $\limsup_I \Delta_i$ is consistent then it is both complete and consistent, and for each $\theta \in \Sent(\sigma)$ we have exactly one of $\theta$ or $\neg \theta$ in $\Th(\limsup_I \Delta_i)$.  If without loss of generality $\theta \in \Th(\limsup_I \Delta_i)$ we get that $\neg \theta$ is eventually not in the $\Delta_i$, say for $j \geq i$ we have $(\neg \theta) \notin \Delta_j$.  By the completeness of the $\Delta_j$ we get that $\theta\in \Delta_j$ for all $j \geq i$.  This shows $\liminf_I\Delta_i$ is complete.

(2) $\Rightarrow$ (3)  Suppose $\liminf_I\Delta_i$ is complete.  If $\theta \notin \liminf_I\Delta_i$ then $(\neg \theta) \in  \liminf_I\Delta_i$, by completeness.  Then for some $i\in I$ we have $(\neg\theta) \in \Delta_j$  for $j\geq i$, and by consistency of the $\Delta_i$ we know that $\theta\notin \Delta_j$ for $j\geq i$.  Thus $\theta \notin \limsup_I \Delta_i$ and we have shown $\limsup_I \Delta_i \subseteq \liminf_I \Delta_i$, so that $\limsup_I\Delta_i = \liminf_I\Delta_i$ and $\lim_I \Delta_i$ exists.

(3) $\Rightarrow$ (1)  If $\lim_I\Delta_i$ exists then $\limsup_I\Delta_i = \liminf_I \Delta_i$ and so $\limsup_I\Delta_i$ is consistent since $\liminf_I \Delta_i$ is consistent.
\end{proof}

The following is an Arzela-Ascoli type result which will be used later in an example.

\begin{proposition}\label{augmentedsubsequence}  Suppose $\{\Delta_i\}_{i\in I}$ is a nonempty collection of theories indexed by a countable directed set $I$.  If $\sigma$ is a countable signature (i.e. $\card(\mathcal{F}),\card(\mathcal{R})\leq \aleph_0$) there exists an unbounded subset $I' \subseteq I$ which is linear under the restricted partial order such that $\lim_{I'}\Delta_i$ exists.  If in addition we have $\Delta_i = \Th(\So_i)$ and that $\So$ is the direct limit of the $\So_i$ with $\So$ countable then $I'$ can be chosen so that $\lim_{I'} \Th(\So_i^*)$ exists.
\end{proposition}

\begin{proof}  We prove the first claim, and the proof of the claim regarding augmented theories follows the same lines.  If $I$ has a maximal element $j$ then letting $I' = \{j\}$ it is easy to see that the limit exists.  Else it is easy to inductively construct a set $I_1$ which is linearly ordered, order isomorphic to the natural numbers $\omega$, and unbounded in $I$.

Since $\sigma$ is countable we know $\Sent(\sigma)$ is countable.  Let $\{\theta_n\}_{n\in \mathbb{N}} = \Sent(\sigma)$.  Let $i_1 =\min(I_1)$.  Supposing we have already defined $I_1, \ldots, I_m$ and $i_1, \ldots, i_m$, we note that at least one of the three sets 

\begin{center}

$T_{m, 0} = \{i\in I_m: i>i_m \wedge\theta_m \in \Delta_i\}$

$T_{m, 1} = \{i\in I_m: i>i_m \wedge (\neg \theta_m) \in \Delta_i\}$

$T_{m, 2} = \{i\in I_m: i>i_m \wedge \theta_m\notin \Delta_i \wedge (\neg \theta)\notin \Delta_i\}$
\end{center}

\noindent is infinite, so let $I_{m+1} = T_{m, n}$ where $n$ is minimal with $T_{m, n}$ infinite.  Let $i_{m+1} = \min(I_{m+1})$.  It is straightforward to check that $I' = \{i_m\}_{m\in \mathbb{N}}$ satisfies the conclusion.
\end{proof}

Recall that a theory $\Delta$ is \textbf{finitely axiomatizable} provided there exists a finite $\Gamma \subseteq \Sent(\sigma)$ such that $\Th(\Gamma) = \Delta$.  We say $\Delta$ is finitely axiomatizable modulo $\Sigma \subseteq \Delta$ provided there exists a finite $\Gamma \subseteq \Sent(\sigma)$ such that $\Th(\Gamma \cup \Sigma) = \Delta$.  Thus $\Delta$ is finitely axiomatizable if and only if it is finitely axiomatizable modulo $\emptyset$.

\begin{proposition}\label{sneakingup}  If $\Delta$ is a complete theory and there exists a directed system of theories $\{\Delta_i\}_{i\in I}$ and $\Sigma \subseteq \Sent(\sigma)$ such that

\begin{enumerate}  \item $(\forall i\in I)[\Sigma \subseteq \Delta_i]$

\item $(\forall i\in I)[\Delta_i \neq \Delta]$

\item $\lim_I\Delta_i = \Delta$

\end{enumerate}

then $\Sigma\subseteq \Delta$ and $\Delta$ is not finitely axiomatizable modulo $\Sigma$.
\end{proposition}

\begin{proof}  Obviously $\Sigma \subseteq \liminf_I\Delta_i = \Delta$ by (1) and (3).  If $\Delta$ were finitely axiomatizable modulo $\Sigma$ we would get a finite $\Gamma\subseteq \Sent(\sigma)$ with $\Delta = \Th(\Gamma \cup \Sigma)$.  As $\lim_I\Delta_i = \Delta$ we then select $i\in I$ with $j \geq i$ implying $\Sigma\cup \Gamma \subseteq \Delta_j$ and therefore $\Delta \subseteq \Delta_j$.  By (2) we know that the inclusion $\Delta \subseteq \Delta_j$ is proper for all $j \geq i$, and since $\Delta$ is complete we see that $\Delta_i$ is inconsistent.  By (3) we have $\Delta$ inconsistent as well, but then $\Delta = \Delta_i$, contradicting (2).
\end{proof}

We similarly say that a theory $\Delta$ is $\lambda$ -\textbf{axiomatizable} (modulo $\Sigma \subseteq \Delta$) provided there exists $\Gamma \subseteq\Sent(\sigma)$ with $\card(\Gamma) \leq \lambda$ and $\Delta = \Th(\Sigma)$ (resp. $\Delta = \Th(\Gamma \cup \Sigma)$).

Proposition \ref{sneakingup} has a converse:

\begin{proposition} \label{sneakingupotherdirection}  Let $\lambda$ be an infinite cardinal.  If $\Delta$ is a theory which is $\lambda$-axiomatizable modulo $\Sigma \subseteq \Delta$ and not $\kappa$-axiomatizable modulo $\Sigma$ for any $\kappa < \lambda$ then there exist $\{\Delta_{\alpha}\}_{\alpha <\lambda}$ such that 

\begin{enumerate}  \item $(\forall\alpha<\lambda)[\Sigma \subseteq \Delta_{\alpha}]$

\item $(\forall \alpha<\lambda)[\Delta_{\alpha} \neq \Delta]$

\item $\lim_{\lambda}\Delta_{\alpha} = \Delta$

\end{enumerate}

Moreover if $\Delta$ is complete the $\Delta_{\alpha}$ may be chosen to be complete.

\end{proposition}

\begin{proof}   Let $\Gamma \subseteq \Delta$ be of cardinality $\lambda$ such that $\Th(\Gamma \cup \Sigma) = \Delta$ and write $\Gamma = \{\theta_{\beta}\}_{\beta<\lambda}$.  Letting $\Delta_{\alpha}= \Th(\{\theta_{\beta}\}_{\beta \leq \alpha} \cup \Sigma)$ we easily get (1) - (3).  If $\Delta$ is complete then for each $\alpha<\lambda$ we select $\phi_{\alpha}\in \Delta \setminus \Th(\{\theta_{\beta}\}_{\beta\leq \alpha} \cup \Sigma)$.  Now let $\Delta_{\alpha}$ be a complete theory including $\{(\neg \phi_{\alpha})\} \cup \{\theta_{\beta}\}_{\beta\leq \alpha} \cup \Sigma$.
\end{proof}

\end{section}

\begin{section}{Direct Limits}\label{directlim}

We give an example of a structure which is a direct limit of substructures such that the limit of the theories of the substructures does not exist.
\begin{example}\label{nolimit}  Consider the relational structure $\So = (\mathbb{N}, \leq, E)$ where $\leq$ is the standard ordering on $\mathbb{N}$ and $Ex$ if and only if $x$ is even.  The finite substructures $\So_n$ where $S_n = \{1, 2, \ldots, n\}$ satisfy

\begin{center}
$(\exists x)[Ex\wedge (\forall y)[y\leq x]]$
\end{center}

\noindent if and only if $n$ is even.  Thus $\lim_{\mathbb{N}}\Th(\So_n)$ does not exist.

\end{example}

Now we list some rather trivial conditions under which the limits exist and conditions  under which the limits are equal to the theory of the direct limit.

\begin{proposition}\label{duh}  The following hold:

\begin{enumerate}

\item  If the elements of $\{\So_i\}_{i\in I}$ are eventually elementarily equivalent (i.e. $\exists i\in I$ such that $j\geq i$ implies $\Th(\So_i) = \Th(\So_j)$) then $\lim_I \Th(\So_i)$ exists.

\item If the elements of $\{\So_i\}_{i\in I}$ are eventually elementarily equivalent to $\So$ then $\lim_I \Th(\So_i) = \Th(\So)$.

\item If the elements of $\{\So_i\}_{i\in I}$ are eventually elementary substructures of $\So$ then $\lim_I \Th(\So_i^*) = \Th(\So^*)$.

\item If $\So$ is the direct limit of $\{\So_i\}_{i\in I}$ and the $\{\So_i\}_{i\in I}$ are eventually elementary substructures in each other ($\exists i\in I$ such that $i_1 \geq i_2 \geq i$ implies $\So_{i_2}$ is an elementary substructure of $\So_{i_1}$) then $\lim_I \Th(\So_i^*) = \Th(\So^*)$.
\end{enumerate}

\end{proposition}

We give an example of a structure which is a direct limit of substructures such that the limit of the first order theories of the substructures exists but is not equal to the theory of the structure.

\begin{example}\label{ZinQ}  Let $\mathcal{S} = (\mathbb{Q}, +, -, 0)$ and substructure $\mathcal{S}_n$ be the cyclic subgroup $\langle \frac{1}{(p_1\cdots p_n)^n}\rangle$ where $\{p_1, \ldots\}$ is an enumeration of the prime numbers.  Each $\mathcal{S}_n$ is isomorphic to $\mathbb{Z}$, so $\Th(\mathcal{S}_n) = \Th(\mathbb{Z})$ and $\lim_{\mathbb{N}} \Th(\mathcal{S}_n) =  \Th(\mathbb{Z})$.  However, $\Th(\mathcal{S})$ has the sentence $(\forall x)(\exists y)[y+y = x]$, so $\lim_{\mathbb{N}} \Th(\mathcal{S}_n) \neq \Th(\mathcal{S})$.
\end{example}

\begin{example}\label{poset}  For each $(n, m)\in \mathbb{N} \times \mathbb{N}$ we let $L_{(n, m)}$ be a linear order of cardinality $m$ and for each $n\in \mathbb{N}$ we let $L_{(n, \infty)}$ be a linear order isomorphic to $\mathbb{N}$.  For each $k\in \mathbb{N}$ let $L_{(\infty, k)}$ be the natural linear order on $\{1, \ldots, k\}$ and let $L_{(\infty, \infty)}$ be the linear order on $\mathbb{N}$.  Thus we consider $L_{(\infty, 1)}\subseteq L_{(\infty, 2)}\subseteq \cdots \subseteq L_{(\infty, \infty)}$.

Let $\bigsqcup_{n\in \mathbb{N}, m\in \mathbb{N} \cup\{\infty\}}L_{(n, m)}$ be given the partial order which preserves the order of each $L_{(n, m)}$ and which has elements of $L_{(n_1, m_1)}$ incomparable with those of $L_{(n_2, m_2)}$ if $(n_1, m_1) \neq (n_2, m_2)$.  For each $k\in \mathbb{N} \cup \{\infty\}$ let $L_k = L_{(\infty, k)} \sqcup\bigsqcup_{n\in \mathbb{N}, m\in \mathbb{N} \cup\{\infty\}}L_{(n, m)}$ also be given the smallest partial order preserving the orders on $L_{(\infty, k)}$ and on $\bigsqcup_{n\in \mathbb{N}, m\in \mathbb{N} \cup\{\infty\}}L_{(n, m)}$.  We have that $L_1 \subseteq L_2\subseteq \cdots \subseteq L_{\infty}$.  Moreover, it is easy to see that $L_{k}$ is isomorphic to $L_{k'}$ for $k, k'\in \mathbb{N} \cup\{\infty\}$.  Thus we can write $\lim_{\mathbb{N}} \Th(L_k) = \Th(L_{\infty})$.  However, letting $c$ denote the minimal element in $L_{(\infty, k)}$ for $k\in \mathbb{N} \cup\{\infty\}$, we see that for every $k\neq \infty$ that the mixed sentence

\begin{center}

$(\exists x)[x\geq c \wedge \neg (\exists y)[y > x]]$

\end{center}

\noindent is in $\Th(L_k^*)$ and is not in $\Th(L_{\infty}^*)$.  Moreover, by Proposition \ref{augmentedsubsequence} we may pass to a subsequence $I'\subseteq \mathbb{N}$ if necessary and get that $\lim_{I'} \Th(L_k^*)$ exists.   Thus we have $\lim_{I'}\Th(L_k) = \Th(L_{\infty})$, that $\lim_{I'}\Th(L_k^*)$ exists and that $\lim_{I'}\Th(L_k^*)\neq \Th(L_{\infty}^*)$.
\end{example}

\begin{example}\label{denseorder}  Let $I = \{A \subseteq \mathbb{R}: \card(A) \leq \aleph_0\}$ with each element endowed with the order inherited from $\mathbb{R}$.  Certainly the structure $\So = (\mathbb{R}, \leq)$ is the direct limit of the substructures in the set $I$.  For any $A\in I$ with $A \supseteq \mathbb{Q}$ we know that $A$ is an elementary substructure of $\So$, so that $\lim_I \Th(A^*) = \Th(\mathbb{R}^*)$ (by the Tarski-Vaught test of elementary subsumption).
\end{example}

\begin{example}\label{freegroups}  Let $\So$ be the free group $F(\{a_i\}_{i=1}^{\infty})$ of countably infinite rank and let $\So_n$ be the subgroup $F(\{x_i\}_{i=1}^{n+1})$ for each $n \geq 1$.  By \cite{KM} or \cite{S1} we have that $\So_n$ is an elementary subgroup of $\So_{n+1}$ and so $\lim_{\mathbb{N}} \Th(\So_n^*) = \Th(\So^*)$.
\end{example}

The following allows you to treat some examples not covered by Proposition \ref{duh}:

\begin{proposition}\label{duh2}  Let $\{\So_i\}_{i\in I}$  have direct limit $\So$.  Suppose for all $\{a_0, \ldots, a_n\}\subseteq S$ there is $i\in I$ such that $a_0, \ldots, a_n\in S_i$  and $j\geq i$ implies there exists an isomorphism $f: \So \rightarrow \So_j$ fixing $\{a_0, \ldots, a_n\}$.  Then $\lim_I \Th(\So_i^*) = \Th(\So^*)$.
\end{proposition}

\begin{proof}  Trivial.
\end{proof}

\begin{example}  Let $\So = (\mathbb{Z}, \leq)$ and $\So_m = (2\mathbb{Z} \cup (\mathbb{Z}\cap [-2m, 2m]), \leq)$.  Let $\{a_1, \ldots, a_n\} \subseteq \mathbb{Z}$ and pick $m$ such that $\{a_1, \ldots, a_n\} \subseteq [-2m, 2m]$.  Letting $k \geq m$ we let $f: \mathbb{Z} \rightarrow S_k$ be the unique isomorphism such that $f| [-2m, 2m]$ is identity.

Notice that none of the $\So_m$ are elementary substructures of $\So$ or of each other.  The sentence $(\exists x)[x\neq 2m\wedge x\neq 2m+2 \wedge 2m \leq x \wedge x\leq 2m+2]$ is true in $\So$ and in $\So_k$ for $k>m$ and false in $\So_k$ for $k\leq m$.  Nevertheless $\lim_I \Th(\So_i^*) = \Th(\So^*)$ by Proposition \ref{duh2}.
\end{example}

If the theory of $\So$ is very thoroughly understood then one might be able to verify that the theory of a directed set of substructures limits to the theory of $\So$, as happens in the following example which provides a contrast to Example \ref{ZinQ}.  Recall that a group $G$ is \textbf{divisible} if for all $g\in G$ and $n\in \mathbb{N}$ there exists $h\in G$ with $h^n = g$.

\begin{example}\label{sneakinguptoQ}  Let $\{p_1, \ldots\}$ be an enumeration of the primes and let $\So_n$ be the additive group $\mathbb{Z}[\frac{1}{(p_1\cdots p_n)^n}] \leq \mathbb{Q}$ under the abelian group signature $(+, -, 0)$ and $\So = (\mathbb{Q}, +, -, 0)$.  We'll use additive group theory notation instead of multiplicative.  The theory of nontrivial torsion-free divisible abelian groups is complete (such groups are $\kappa$-categorical for all uncountable $\kappa$, and no such group is finite so we have completeness of the theory by the Vaught completeness test \cite{V}).  Then $\Th(\mathbb{\So})$ is completely axiomatized by the abelian group axioms, plus the torsion-free axioms:

\begin{center}
$(\forall x)[x\neq 0 \Rightarrow mx \neq 0]$
\end{center}

\noindent for each $m\in \mathbb{N}$, plus the divisibility axioms:

\begin{center}

$(\forall x)(\exists y)[my = x]$

\end{center}

\noindent for each $m\in \mathbb{N}$, plus the nontriviality axiom $(\exists x)[x\neq 0]$.  The models $\So_n$ all satisfy the abelian group, torsion-free and nontriviality axioms.  Moreover the divisibility axiom $(\forall x)(\exists y)[my = x]$ is satisfied in $\So_n$ for all $n\geq k$, where $m$ divides $(p_1\cdots p_k)^k$.

Thus $\liminf_{\mathbb{N}}\Th(\So_n)\supseteq \Th(\So)$ and so $\liminf_{\mathbb{N}}\Th(\So_n) \supseteq \Th(\So)$ is complete and by Lemma \ref{trivial3} we get that $\lim_{\mathbb{N}}\Th(\So_n)$ exists and is consistent.  As $\Th(\So)$ is complete we obtain $\lim_{\mathbb{N}}\Th(\So_n)= \Th(\So)$.
\end{example}

We restate and prove Theorem \ref{manyautomorphisms}.

\begin{A}  Suppose $\{\So_i\}_{i\in I}$ is a collection of substructures of $\So$ such that $\lim_I \So_i = \So$.
\begin{enumerate}

\item  Suppose that for each $m \in \mathbb{N}$ and $i' \in I$ there exists $i \geq i'$ such that for all $j \geq i$ and $(m_1, m_2) \in \mathbb{N}^2$ such that $m_1 + m_2 \leq m$ and $c_1, \ldots, c_{m_0} \in S_i$ and $d_1, \ldots, d_{m_2} \in S_j$ there exists an automorphism $f: \So_j \rightarrow \So_j$ with $f(c_l) = c_l$ and $f(\{d_1, \ldots, d_{m_2}\}) \subseteq S_i$.  Then the limit $\lim_I \Th(\So_j^*)$ exists.

\item  If, in addition to (1), for each $m \in \mathbb{N}$ and $i' \in I$ there exists $i \geq i'$ such that whenever $(m_1, m_2) \in \mathbb{N}^2$ are such that $m_1 + m_2 \leq m$ and $c_1, \ldots, c_{m_0} \in S_i$ and $d_1, \ldots, d_{m_2} \in S$ there exists an automorphism $f: \So \rightarrow \So$ with $f(c_l) = c_l$ and $f(\{d_1, \ldots, d_{m_2}\}) \subseteq S_i$.  Then $\lim_I \Th(\So_j^*) = \Th(\So^*)$.

\end{enumerate}
\end{A}

\begin{proof}  Let $\phi(a_1, \ldots, a_p)$ be a sentence in $L_{\sigma, \mathcal{S}^*}$.  Without loss of generality let $\phi$ be in prenex form 
\begin{center}
$\phi(a_1, \ldots, a_p) \equiv \exists x_1 \cdots\exists x_{n_1}\forall x_{n_1 + 1} \cdots \forall x_{n_{2}} \cdots \cdots  \exists x_{n_k}\psi(a_1, \ldots, a_p, x_1, \ldots, x_{n_k})$

\end{center}

\noindent Let $m =  p + n_k$.  Select $i' \in I$ large enough that $a_1, \ldots a_p \in S_{i'}$.  Select $i$ as in the conclusion of (1) and let $j \geq i$ be given.  We show that if $\phi$ holds in $\So_j$ then $\phi$ also holds in $\So_i$.  If $\phi$ does not hold in $\So_j$ then the same proof shows that $\phi$ does not hold in $\So_i$ since $\neg \phi$ is logically equivalent to a sentence in prenex form which uses only the constants $a_1, \ldots, a_p$ and has $n_k$ quantifiers.  Thus the statement $\phi$ is either true in $\So_j$ for all $j \geq i$ or is false in $\So_j$ for all $j\geq i$.

Assuming $\phi$ holds in $\So_j$ we select constants $b_1, \ldots, b_{n_1}$ in $S_j$ so that

\begin{center}
$\forall x_{n_1+1} \cdots \forall x_{n_{2}} \cdots \cdots  \exists x_{n_k}\psi(a_1, \ldots, a_p, b_1, \ldots, b_{n_1}, x_{n_1+1}, \ldots x_{n_k})$
\end{center}

\noindent holds in $S_j$.  Select an automorphism $f_1:\So_j\rightarrow \So_j$ which fixes each of $a_1, \ldots, a_p$ and such that $f_1(\{b_1, \ldots, b_{n_1}\}) \subseteq S_i$.  Then the sentence 

\begin{center}
$\forall x_{n_1+1} \cdots \forall x_{n_{2}} \cdots \cdots  \exists x_{n_k}\psi(a_1, \ldots, a_p, f_1(b_1), \ldots, f_1(b_{n_1}), x_{n_1+1}, \ldots x_{n_k})$
\end{center}

\noindent holds in $S_j$ as well.  Let $b_{n_1+1}, \ldots, b_{n_2}\in S_i$ be given.  Since the sentence

\begin{center}
$\exists x_{n_2+1} \cdots \exists x_{n_{3}}\forall x_{n_3+1} \cdots \cdots  \exists x_{n_k}\psi(a_1, \ldots, a_p, f_1(b_1), \ldots, f_1(b_{n_1}), b_{n_1+1}, \ldots, b_{n_2}, \ldots,  x_{n_k})$
\end{center}

\noindent is true in $\So_j$ we may select $b_{n_2+1}, \ldots, b_{n_3}\in S_j$ so that 

\begin{center}
$\forall x_{n_3+1} \cdots \cdots  \exists x_{n_k}\psi(a_1, \ldots, a_p, f_1(b_1), \ldots, f_1(b_{n_1}), b_{n_1+1}, \ldots, b_{n_2}, b_{n_2+1}, \ldots, b_{n_3},$

$x_{n_3 + 1}, \ldots,  x_{n_k})$
\end{center}

\noindent holds in $S_j$.  Select an automorphism $f_3:\So_j \rightarrow \So_j$ which fixes all of $a_1, \ldots, a_p$, $f_1(b_1), \ldots, f_1(b_{n_1})$, $b_{n_1+1}, \ldots, b_{n_2}$ and such that $f_3(\{b_{n_2+1}, \ldots, b_{n_3}\}) \subseteq S_i$.  Now 

\begin{center}
$\forall x_{n_3+1} \cdots \cdots  \exists x_{n_k}\psi(a_1, \ldots, a_p, f_1(b_1), \ldots, f_1(b_{n_1}), b_{n_1+1}, \ldots, b_{n_2}, f_3(b_{n_2+1}), \ldots, f_3(b_{n_3}),$

$x_{n_3+1}, \ldots,  x_{n_k})$
\end{center}

\noindent holds in $\So_j$, so we let $b_{n_3+1}, \ldots, b_{n_4} \in S_i$ be given.  Continue selecting constants for existential claims that hold in $\So_j$, mapping them into $S_i$ via automorphisms which fix all previously determined constants known to be in $S_i$, and letting elements of $S_i$ be selected where we encounter universal quantifiers.  Once we have exhausted all the quantifiers of $\phi$ we have indeed checked that $\phi$ holds in $\So_i$.

To prove (2) we assume the hypotheses.  Since the hypotheses of (1) hold we know that $\lim_I \Th(\So_i^*)$ exists and so it is sufficient to show that for arbitrarily high $j\in I$ we have $\phi$ true in $\So$ implies $\phi$ true in $\So_j$.  For this we let $\phi(a_1, \ldots, a_p)$ be given, again assuming without loss of generality that $\phi$ is in prenex form with the number of quantifiers being $n_k$ and let $m = p + n_k$.  As $\lim_I \Th(\So_i^*)$ exists we let $i' \in I$ be such that for all $j \geq i'$ we have $\phi \in \Th(\So_{i'}^*)$ if and only if $\phi\in \Th(\So_j^*)$.  Pick $i$ as in the statement of (2).  By the same proof as that for (1) we determine that $\phi\in \Th(\So^*)$ if and only if $\phi \in \Th(\So_i^*)$.  Thus we conclude that $\lim_I \Th(\So_j^*) = \Th(\So^*)$.
\end{proof}

This theorem has applications in structures which have lots of automorphisms.  We give some examples.

\begin{example}\label{infinitesets}  Consider an infinite set under the trivial signature.  Thus $\So$ is simply an infinite abstract set and the language $L_{\sigma, \So^*}$ is the language of predicate calculus with equality and with infinitely many constants.  Let $I$ index the partially ordered set of finite subsets of $S$.  For each $m\in \mathbb{N}$ and $i' \in I$ we pick $i \geq i'$ such that $|S_i| \geq m$.  For any $j \geq i$ and $c_1, \ldots, c_{m_1} \in S_i$ and $d_1, \ldots, d_{m_2} \in S_j$ there is a bijection $f: S_j \rightarrow S_j$ which fixes $c_1, \ldots, c_{m_1}$ and has $f(\{d_1, \ldots, d_{m_2}\}) \subseteq S_i$.  Thus condition (1) of Theorem \ref{manyautomorphisms} holds, and condition (2) is similarly verified.  Thus $\lim_I \Th(\So_i^*) = \Th(\So^*)$ and we conclude that the augmented first order theory of an infinite set is the limit of the augmented first order theories of its finite subsets.
\end{example}

\begin{example}\label{infinitedimensionalvectorspaces}  Let $F$ be a field and let $\So$ be an infinite dimensional vector space over $F$.  The signature on $\So$ is that of $F$-vector spaces, with an abelian group operation $+$, additive inverse $-$, identity $0$, and for each element of $F$ a unary function giving scalar multiplication.  Let $I$ index the collection of finite dimensional subspaces of $\So$.  Given any $m\in \mathbb{N}$ and $i' \in I$ we select $i \geq i'$ with $\So_i$ of dimension $\geq m$.  If $\So_j \geq \So_i$ and $c_1, \ldots c_{m_1} \subseteq S_i$ and $d_1, \ldots, d_{m_2} \subseteq S_j$ we let $W_1$ be the span of $\{c_1, \ldots, c_{m_1}\}$ and $W_2 \leq \So_j$ be such that $\So_i = W_1 \oplus W_2$.  Pick a subspace $W_3 \leq S_j$ such that $W_1 \oplus W_3$ is the span of $\{c_1, \ldots, c_{m_1}, d_1, \ldots, d_{m_2}\}$.  As $W_3$ is of dimension $\leq m_2$ and $W_2$ is of dimension $\geq m_2$ there exists a nonsingular linear transformation $f:\So_j \rightarrow \So_j$ which fixes $W_1$ and satisfies $f(W_3) \leq W_2$.  Thus condition (1) of Theorem \ref{manyautomorphisms} holds and a similar check shows that condition (2) holds.
\end{example}

\begin{example}\label{finitefieldvectorspaces}  Let $\So$ be an infinite abelian group of exponent $p$ (i.e. $(\forall g\in S)[pg = 0]$).  Then $\So$ is a $\mathbb{Z}/p$ vector space and letting $I$ index the collection of finite subgroups we get that $\lim_I \Th(\So_i^*) = \Th(\So^*)$ as a reduct of the previous example.
\end{example}

\begin{example}\label{infiniterankfreeabeliangroups}  Let $\So$ be a free abelian group of infinite rank and let $I$ index the collection of all finite rank subgroups of $\So$ which are direct summands of $\So$ (i.e. there exists $C\leq \So$ with $\So = A \oplus C$.)  We shall make use of the following fact from abelian group theory (which is an immediate consequence of Theorem 1.1 of \cite[Section II.1]{H}):

\begin{fact*}\label{breakingoffsummand}  If $A$ is a free abelian group and $B \leq A$ is of finite rank then there exist $B\leq B' \leq A$ and $C\leq A$ with $B'$ of the same rank as $B$ and $A = B'\oplus C$.
\end{fact*}

Let $m\in \mathbb{N}$ and $i' \in I$ be given.  Select $i \geq i'$ with $\So_i$ of rank $\geq 3m$.  Let $m_1 + m_2 \leq m$, $j \geq i$, $c_1, \ldots, c_{m_1} \in S_i$ and $d_1, \ldots, d_{m_2} \in S_j$.  Select $A_1 \leq \So$ with $\So = \So_i \oplus A_1$, and let $\pi_{\So_i}:\So \rightarrow \So_i$ be projection according to the direct sum.  Now the kernel $\ker(\pi_{\So_i}|\So_j) = A$ of the restriction $\pi_{\So_i}|\So_j$ is a subgroup of $\So_j$ such that $\So_j = \So_i \oplus A$.  Let $\pi_A:\So_j \rightarrow A$ be the projection determined by $\pi_A(a) = a - \pi_{\So_i}(a)$.  Notice that the subgroup $\langle c_1, \ldots, c_{m_1}, \pi_{\So_i}(d_1), \ldots,  \pi_{\So_i}(d_{m_2})\rangle \leq \So_i$ is of rank at most $2m$ and so there exist subgroups $C, D \leq \So_i$ such that $C\geq \langle c_1, \ldots, c_{m_1}, \pi_{\So_i}(d_1), \ldots,  \pi_{\So_i}(d_{m_2})\rangle$ is of rank at most $2m$ and $C \oplus D = \So_i$.  Similarly we notice that $\langle \pi_A(d_1), \ldots, \pi_A(d_{m_2})\rangle \subseteq A$ is of rank at most $m$ and so there exist subgroups $E, F\leq A$ such that $E \geq \langle \pi_A(d_1), \ldots, \pi_A(d_{m_2})\rangle$, $E$ has rank at most $m$, and $A = E\oplus F$.

Thus $\So_j = C\oplus D\oplus E\oplus F$ with 

\begin{center}
$\So_i = C\oplus D$

$A = E\oplus F$

$C\geq \langle c_1, \ldots, c_{m_1}, \pi_{\So_i}(d_1), \ldots,  \pi_{\So_i}(d_{m_2})\rangle$

$E \geq \langle \pi_A(d_1), \ldots, \pi_A(d_{m_2})\rangle$
\end{center}

\noindent and $C$ of rank at most $2m$ and $E$ of rank at most $m$.  Since $\So_i$ has rank at least $3m$ we get that $D$ has rank at least $m$.  Now take $f:\So_j \rightarrow \So_j$ to be any isomorphism which fixes $C$ and takes $f(E) \leq D$.  This establishes condition (1) of Theorem \ref{manyautomorphisms} and condition (2) is established similarly.
\end{example}

From the above examples we obtain some intuitive aphorisms:

\vspace{.2in}
\noindent\textit{The first order theory of infinite sets is the limit of the first order theory of increasingly large finite sets.}

\vspace{.2in}
\noindent\textit{The (augmented) first order theory of an infinite abelian group of prime exponent $p$ is the limit of the first order theory of increasingly large finite abelian subgroups of exponent $p$.}

\vspace{.2in}

\noindent\textit{The (augmented) first order theory of a free abelian group of infinite rank is the limit of the first order theory of increasingly large finite rank direct summands.}

\vspace{.2in}

From Examples \ref{sneakinguptoQ}, \ref{infinitesets}, \ref{finitefieldvectorspaces} and \ref{breakingoffsummand} we get an easy proof of the following corollary.  Parts (1), (2) and (3) follow without any recourse to high power theorems (like model theory compactness).  Part (4) involves the discussion of Example \ref{sneakinguptoQ} which used a consequence of the L\"owenheim-Skolem Theorem and the selection of a Hamel basis for an uncountable vector space over $\mathbb{Q}$.

\begin{corollary}\label{notfinitelyax}  The following hold:

\begin{enumerate}
\item The property of being infinite is not finitely axiomatizable.

\item The property of being an infinite abelian group of exponent $p$ is not finitely axiomatizable.

\item The theory of an infinite rank free abelian group is not finitely axiomatizable modulo the theory of torsion-free abelian groups.

\item  The property of being a nontrivial torsion-free divisible abelian group is not finitely axiomatizable modulo the theory of nontrivial torsion-free abelian groups.

\end{enumerate}

\end{corollary}

\begin{proof}  For (1) and (2) apply Proposition \ref{sneakingup}, noticing that for any finite set or finite abelian group of exponent $p$ there is some sentence 

\begin{center}
$(\exists x_1)\cdots (\exists x_m)[\bigwedge_{1\leq k \neq  l \leq m}x_k \neq x_l]$
\end{center}

\noindent which is true in an infinite model but not in the finite model.  For (3) we notice that the sentence

\begin{center}  $(\forall x_1)\cdots (\forall x_m)[\bigvee_{\emptyset \neq T \subseteq \{1, \ldots, m\}}(\exists z)[2z = \sum_{i\in T}x_i]]$

\end{center}
 \noindent is true in all free abelian groups of rank $< m$ and false in those of higher rank, and apply Proposition \ref{sneakingup}.  For (4) we use the models $\So_n$ in Example \ref{sneakinguptoQ}.  Letting $m = (p_1\cdots p_n)^n$ the sentence $(\forall x)(\exists y)[mx = y]$ is true in $\So_q$ for $q \geq n$ and false when $q<n$.  Apply Proposition \ref{sneakingup}.
\end{proof}

We end with a question which does not seem be covered by any of the criteria in this note:

\begin{question}  Let $\So$ be the group $S_{\infty}$ of finite supported bijections on $\mathbb{N}$ and $I$ the collection of symmetric subgroups $S_F = \{g\in S_{\infty}: \supp(g) \subseteq F\}$ where $F \subseteq \mathbb{N}$ is finite.  Is it the case that $\lim_I \Th(\So_i^*) = \Th(\So^*)$?
\end{question}

\end{section}

\section*{Acknowledgement}

The author is indebted to an anonymous referee for suggested corrections of errors in an earlier version of this note.

\bibliographystyle{amsplain}

\end{document}